\documentclass[11pt, notitlepage]{amsart}
\usepackage{a4wide}
\usepackage{amsthm,amsmath,amssymb,amscd,mathrsfs,enumerate,amsfonts}

\swapnumbers
\newtheorem{thm}{Theorem}[section]
\newtheorem{prop}[thm]{Proposition}
\newtheorem{lemma}[thm]{Lemma}

\newtheorem{convention}[thm]{Convention}
\theoremstyle{definition}
\newtheorem*{defin}{Definition}

\newtheorem{notation}[thm]{Notation}
\newtheorem*{notaIntr}{Notation}
\theoremstyle{remark}
\newtheorem{remark}[thm]{Remark}

\def\ifff{if and only if }
\def\letsM{Let us fix a $(*)$-elementary submodel $M$}
\def\letsMX{Let us fix a $(*)$-elementary submodel $M$ containing $X$ }

\def\pp{For any suitable elementary submodel $M$ the following holds:}

\def \rng {\operatorname{Rng}}
\def \dom {\operatorname{Dom}}

\def\en{\mathbb N}
\def\zet{\mathbb Z}
\def\er{\mathbb R}
\def\qe{\mathbb Q}
\def\ef{\mathcal F}

\def\iff{\leftrightarrow}
\def\eps{\varepsilon}  
\def\ov{\overline}
\begin{document}
\author{Marek C\'uth, Martin Rmoutil}
\title{$\sigma$-porosity is separably determined}
\thanks{M.C\'uth was supported by the Grant No. 282511/B-MAT/MFF of the Grant Agency of the Charles University in Prague. M. Rmoutil was supported by the Grant SVV-2011-263316.}
\email{cuthm5am@karlin.mff.cuni.cz, caj@rmail.cz}
\address{Charles University, Faculty of Mathematics and Physics, Sokolovsk\'a 83, 186 75 Praha 8 Karl\'{\i}n, Czech Republic}
\begin{abstract}
We prove a separable reduction theorem for $\sigma$-porosity of Suslin sets. In particular, if $A$ is a Suslin subset in a Banach space $X$, then each separable subspace of $X$ can be enlarged to a separable subspace $V$
such that $A$ is $\sigma$-porous in $X$ if and only if $A\cap V$ is $\sigma$-porous in $V$. Such a result is proved for several types of $\sigma$-porosity. The proof is done using the method of elementary submodels, hence the results can be combined with other separable reduction theorems. As an application we extend a theorem of
L.Zaj\'{\i}\v{c}ek on differentiability of Lipschitz functions on separable Asplund spaces to the nonseparable setting.
\end{abstract}
\maketitle
\section{Introduction}
The aim of this article is to obtain separable reduction theorems for some classes of $\sigma$-porous sets by employing the method of elementary submodels. This is a set-theoretical method which can be used in various branches of mathematics. A.~Dow in \cite{dow} illustrated the use of this method in topology, W.~Kubi\'s in \cite{kubis} used it in functional analysis, namely to construct projections on Banach spaces.

In this article we shall use the method of elementary submodels to prove Theorem \ref{tUpPorous} and Theorem \ref{tLoPorous} which have as a consequence for example the following:

\begin{thm}\label{tUpperPorous}
Let $\langle X,\left\| \cdot \right\|\rangle$ be a Banach space and let $A\subset X$ be a Suslin set. Then for every separable subspace $V_0\subset X$ there exists a closed separable space $V\subset X$ such that $V_0\subset V$ and
\begin{enumerate}[(i)] 
  \item $A$ is $\sigma$-upper porous \ifff $A\cap V$ is $\sigma$-upper porous in the space $V$,
  \item $A$ is $\sigma$-lower porous \ifff $A\cap V$ is $\sigma$-lower porous in the space $V$.
\end{enumerate}
\end{thm}

As a consequence of Theorem \ref{tUpPorous} and [\ref{cuth}, Theorem 5.7] we get the following:

\begin{thm}\label{tDifPorous}
Let $X, Y$ be Banach spaces, $G\subset X$ an open subset and $f:G\to Y$ be a function. Then for every separable subspace $V_0\subset X$ there exists a closed separable space $V\subset X$ such that $V_0\subset V$ and that the following two conditions are equivalent:
\begin{enumerate}[(i)] 
  \item the set of the points where $f$ is not Fr\'echet differentiable is $\sigma$-upper porous,
  \item the set of the points where $f\upharpoonright V$ is not Fr\'echet differentiable is $\sigma$-upper porous in $V$.
\end{enumerate}
\end{thm}

The first result is in a certain sense an improvement of the result of J. Lindenstrauss, D.Preiss and J.Ti\v{s}er [\ref{tiser}, Corrolary 3.6.7], from where only one implication follows.
Moreover, we are able to easily extend results concerning points of non-differentiability from separable Banach spaces to the non-separable case. An example of such a result is Theorem \ref{tSepDual} which has been proved in the article \cite{5} -- the generalization is in Theorem \ref{tApplication}.


Let us recall the most relevant notions, definitions and notations:

\begin{notaIntr}
We denote by $\omega$ the set of all natural numbers (including $0$), by $\en$ the set $\omega\setminus\{0\}$, by $\er_+$ the interval $(0,\infty)$ and $\qe_+$ stands for $\er_+\cap\qe$. Whenever we say that a set is countable, we mean that the set is either finite or infinite and countable. If $f$ is a mapping then we denote by $\rng f$ the range of $f$ and by $\dom f$ the domain of $f$. By writing $f:X\to Y$ we mean that $f$ is a mapping with $\dom f = X$ and $\rng f \subset Y$. By the symbol $f\!\upharpoonright_{Z}$ we denote the restriction of the mapping $f$ to the set $Z$.

If $\langle X,\rho\rangle$ is a metric space, we denote by $U(x,r)$ the open ball (i.e. the set $\{y\in X: \rho(x,y) < r\}$) and by $d(x,A)$ the distance function from a set $A\subset X$ (i.e. $d(x,A)=\inf\{\rho(x,a); a\in A\}$). We shall consider normed linear spaces over the field of real numbers (but many results hold for complex spaces as well). If $X$ is a normed linear space, $X^*$ stands for the (continuous) dual space of $X$.
\end{notaIntr}

\section{Elementary Submodels}

The method of elementary submodels enables us to find specific separable subspaces (of Banach spaces) which can be used for proofs of separable reduction theorems. In this section we briefly describe this method and recall some basic notions. More information can be found in \cite{cuth} where this method is described in greater detail.

First, let us recall some definitions:

Let $N$ be a fixed set and $\phi$ a formula in the language of $ZFC$. Then the {\em relativization of $\phi$ to $N$} is the formula $\phi^N$ which is obtained from $\phi$ by replacing each quantifier of the form ``$\forall x$'' by ``$\forall x\in N$'' and each quantifier of the form ``$\exists x$'' by ``$\exists x\in N$''.

For example, if
$$\phi = \forall x\; \forall y\; \exists z\; ((x\in z) \;\wedge\; (y\in z)) $$
and $N=\{a,b\}$, then the relativization of $\phi$ to $N$ is
$$\phi^N = \forall x\in N\; \forall y\in N\; \exists z\in N\; ((x\in z)\; \wedge\; (y\in z)).$$
It is clear that $\phi$ is satisfied, but $\phi^N$ is not.

If $\phi(x_1,\ldots,x_n)$ is a formula with all free variables shown (i.e. a formula whose free variables are exactly $x_1,\ldots,x_n$) then {\em $\phi$ is absolute for $N$} if and only if
$$\forall a_1,\ldots,a_n\in N\quad (\phi^N(a_1,\ldots,a_n) \leftrightarrow \phi(a_1,\ldots,a_n)).$$

The method is based mainly on the following set-theoretical theorem (a proof can be found in [\ref{kunen}, Chapter IV, Theorem 7.8]).

\begin{thm}\label{tCountModel}
 Let $\phi_1,\ldots,\phi_n$ be any formulas and $X$ any set. Then there exists a set $M\supset X$ such, that
 $$(\phi_1,\ldots,\phi_n \text{ are absolute for }M)\quad \wedge\quad (|M|\leq \max(\omega,|X|)).$$ 
\end{thm}

Since the previous theorem will often be used throughout the paper, the following notation is useful.

\begin{defin}
 Let $\phi_1,\ldots,\phi_n$ be any formulas and let $X$ be any countable set. Let $M\supset X$ be a countable set satisfying that $\phi_1,\ldots,\phi_n$ are absolute for $M$. Then we say that {\em $M$ is an elementary submodel for $\phi_1,\ldots,\phi_n$ containing $X$}. This is denoted by $M\prec(\phi_1,...,\phi_n;\; X)$.
\end{defin}

Let $\phi(x_1,\ldots,x_n)$ be a formula with all free variables shown and let $M$ be some elementary submodel for $\phi$. To use the absoluteness of $\phi$ for $M$ efficiently, we need to know that many sets are elements of $M$. The reason is that for $a_1,\ldots,a_n\in M$ we have $\phi(a_1,\ldots,a_n)$ if and only if $\phi^M(a_1,\ldots,a_n)$.  Using the following lemma we can force the elementary submodel $M$ to contain all the required objects created (uniquely) from elements of $M$ (for a proof see [\ref{cuth}, Lemma 2.5]).

\begin{lemma}\label{lUniqueInM}
Let $\phi(y,x_1,\ldots,x_n)$ be a formula with all free variables shown and let $X$ be a countable set. Let $M$ be a fixed set, $M\prec(\phi, \exists y\;\phi(y,x_1,\ldots,x_n);\; X)$ and let $a_1,\ldots,a_n \in M$ be such that there exists only one set $u$ satisfying $\phi(u,a_1,\ldots,a_n)$. Then $u\in M$.
\end{lemma}

It would be very laborious and pointless to use only the basic language of the set theory. For example, we often write $x < y$ and we know that this is in fact a shortcut for the formula $\varphi(x,y,<)$ with all free variables shown. Therefore, in the following text we use this extended language of the set theory as we are used to. We shall also use the following convention.
\begin{convention} \label{conventionM}
 Whenever we say\\[8pt]
 {\em for any suitable elementary submodel $M$ (the following holds...)},\\[8 pt]
 we mean that\\[8pt]
 {\em there exists a list of formulas $\phi_1,\ldots,\phi_n$ and a countable set $Y$ such that for every $M\prec(\phi_1,\ldots,\phi_n;\;Y)$ (the following holds...)}.
\end{convention}

By using this new terminology we lose the information about the formulas $\phi_1,\ldots,\phi_n$ and the set $Y$. This is, however, not important in applications.

\begin{remark}\label{rCombine}
 Let us have sentences $T_1(a),\ldots,T_n(a)$. Assume that whenever an $i\in \{1,\ldots,n\}$ is given, then for any suitable elementary submodel $M_i$ the sentence $T_i(M_i)$ is satisfied. Then it is easy to verify that for any suitable model $M$ the sentence $$T_1(M) \;\wedge\; \ldots \;\wedge\; T_n(M)$$ is satisfied (it suffices to combine all the lists of formulas and all the sets from the definition above).
 
 In other words, we are able to combine any finite number of results we have proved using the technique of elementary submodels. This includes all the theorems starting with ``\pp''
\end{remark}

Let us recall several more results about suitable elementary submodels (proofs can be found in [\ref{cuth}, Chapters 2 and 3]):

\begin{prop}\label{pModelBasis}\pp
  \begin{enumerate}
    \item[(i)] If $A,B\in M$, then $A\cap B\in M$, $B\setminus A\in M$ and $A\cup B\in M$.
    \item[(ii)] Let $f$ be a function such that $f\in M$. Then $\dom{f} \in M$, $\rng{f} \in M$ and for every $x\in \dom{f}\cap M$, $f(x)\in M$.
    \item[(iii)] Let $S$ be a finite set. Then $S\in M$ \ifff $S\subset M$.
    \item[(iv)] Let $S\in M$ be a countable set. Then $S\subset M$.
    \item[(v)] For every natural number $n>0$ and for arbitrary $(n+1)$ sets  $a_0,\ldots,a_n$ it is true,that
				$$a_0,\ldots,a_n\in M \iff \langle a_0,\ldots,a_n\rangle\in M.$$
  \end{enumerate}  
\end{prop}

\begin{notation} 
\mbox{}
\begin{itemize}
\item If $A$ is a set, then by saying that an elementary model $M$ contains $A$ we mean that $A\in M$.
\item If $\langle X,\rho\rangle$ is a metric space (resp. $\langle X, +, \cdot, \left\| \cdot \right\|\rangle$ is a normed linear space) and $M$ an elementary submodel, then by saying {\em $M$ contains $X$} (or by writing $X\in M$) we mean that $\langle X,\rho\rangle\in M$ (resp. $ \langle X, +, \cdot, \left\| \cdot \right\|\rangle\in M$). 
\item If $X$ is a topological space and $M$ an elementary submodel, then we denote by $X_M$ the set $\ov{X\cap M}$.
\end{itemize}
\end{notation}

\begin{prop}\label{pModelBasisSpaces}\pp
  \begin{enumerate}
    \item[(i)] If $X$ is a metric space then whenever $M$ contains $X$, it is true that 
    $$\forall r\in\er_+\cap M\;\forall x\in X\cap M\quad U(x,r)\in M.$$
    \item[(ii)] If $X$ is a normed linear space then whenever $M$ contains $X$, it is true that 
    $$X_M\text{ is closed separable suspace of }X.$$
  \end{enumerate}  
\end{prop}

\begin{convention}
The proofs in the following text often begin in the same way. To avoid unnecessary repetitions, by saying {\em``Let us fix a $(*)$-elementary submodel $M$ [containing $A_1,\ldots,A_n$]''} we will understand the following:\\[8 pt]
{\em Let us have formulas $\varphi_1,\ldots,\varphi_m$ and a countable set $Y$ such that the elementary submodel $M\prec(\varphi_1,\ldots,\varphi_m;\;Y)$ is suitable for all the propositions from \cite{cuth}. Add to them formulas marked with $(*)$ in all the preceding proofs from this paper and formulas marked with $(*)$ in the proof below (and all their subformulas). Denote such a list of formulas by $\phi_1,\ldots,\phi_k$. Let us fix a countable set $X$ containing the sets $Y$, $\omega$, $\zet$, $\qe$, $\qe_+$, $\er$, $\er_+$ and all the common operations and relations on real numbers ($+$, $-$, $\cdot$, $:$, $<$). Fix an elementary submodel $M$ for formulas $\phi_1,\ldots,\phi_k$ containing $X$ [such that $A_1,\ldots,A_n\in M$].}
\end{convention}
Thus, any $(*)$-elementary submodel $M$ is suitable for the results from \cite{cuth} and all the preceding theorems and propositions from this paper, making it possible to use all of these results for $M$.

In order to demonstrate how this technique works, we prove the following two easy lemmas which we use later (the proof of the second lemma is also  contained in the proof of Proposition 4.1 in \cite{cuth}).
\begin{lemma}\label{lNonempty}\pp\\ Whenever $A\in M$ is a nonempty set, then $A\cap M$ is nonempty. 
\end{lemma}
\begin{proof}
\letsM and fix some nonempty set $A\in M$. Then
 \begin{flalign*}
    (*) & & \exists x\quad (x\in A). & &
 \end{flalign*}
 This formula has only one free variable $A$ and the set $A$ is contained in $M$. Thus, due to the absoluteness of the formula above, there exists an $x\in M$ such that $x\in A$.
\end{proof}

\begin{lemma}\label{lDense}\pp\\ Let $\langle X,\rho\rangle$ be a metric space, $B\subset X$. Then whenever $M$ contains $X$, $B$ and a set $D\subset B$, it is true that
 $$D\text{ is dense in }B \rightarrow D\cap M\text{ is dense in }B\cap X_M.$$
\end{lemma}
\begin{proof}
 \letsMX such that $B,D\in M$. If the set $B$ is empty then the proposition is obvious. Otherwise fix $b\in B\cap X_M$ and $r > 0$. Choose some $b_o\in U(b,\frac r2)\cap M$ and a rational number $q\in (\rho(b,b_0),\tfrac r2)$. Then $U(b_0,q)\subset U(b,r)$ and
  \begin{flalign*}
    (*) & & \exists d\in D\quad (d\in U(b_0,q)). & &
  \end{flalign*}
 In the preceding formula we use the shortcut $d \in U(b_0,q)$ which stands for $d\in X \;\wedge\; \rho(d,b_0) < q$. Free variables in this formula are $X,\rho, <, D, b_0, q$. Those are contained in $M$ and thus we can use the absoluteness to find a $d\in D\cap M$ such that $(d\in U(b_0,q))^M$. Using the absoluteness again we obtain that $d$ is an element of $U(b_0,q)$. Consequently, 
$$U(b,r)\cap D\cap M\supset U(b_0,q)\cap D\cap M\neq \emptyset$$ 
and so the set $D\cap M$ is dense in $B\cap X_M$.
\end{proof}

\section{$\sigma$-porous sets}
In this section we compile several known results concerning different notions of $\sigma$-porous sets. The usefulness of these facts for our needs will be apparent later; for more information about properties and applications of different types of porosity we refer the reader to survey articles \cite{7} and \cite{zajicPory}. On some occasions we shall also refer to the paper \cite{1}.

Let us begin by stating several basic definitions.

\begin{defin}\label{dPorosities}
Let $\langle X, \rho\rangle$ be a metric space, $A\subset X$, $x \in X$ and $R > 0$. Then we denote by $\gamma(x,R,A)$ the supremum of all $r\geq 0$ for which there exists $z\in X$ such that $U(z,r)\subset U(x,R)\setminus A$. The set $A$ is called \emph{upper porous at} $x$ {\em in the space $X$} if
$$\limsup_{R\to 0^+}\frac{\gamma(x,R,A)}{R} > 0.$$

In most cases it is clear which space $X$ we have in mind. Therefore we often omit the words ``{\em in the space $X$}". (We shall apply this convention to other notions as well.) 

Let $g$ be a strictly increasing real-valued function defined on $[0,h)$ (where $h>0$) with $g(0)=0$. Such a function is called \emph{porosity function}. We say that $A$ is \emph{$\langle g \rangle$-porous at $x$ (in the space $X$)} if there exists a sequence of open balls $\{U(c_n,r_n)\}$ such that $c_n\to x$, $U(c_k,r_k)\cap A=\emptyset$ and $x\in U(c_k, g(r_k))$ for each $k$.

We say the set $A$ is \emph{$\langle g \rangle$-porous} (\emph{upper-porous}) if it is \emph{$\langle g \rangle$-porous} (\emph{upper-porous}) at each of its points and \emph{$\sigma$-$\langle g \rangle$-porous} (\emph{$\sigma$-upper porous}) if it is a countable union of \emph{$\langle g \rangle$-porous} (\emph{upper-porous}) sets.
\end{defin}

\begin{defin}
 Let $\langle X, \rho \rangle$ be a topologically complete metric space and let $g$ be a porosity function. We say that $\ef$ is a \emph{Foran system for $\langle g \rangle$-porosity} \emph{in} $X$ if the following conditions hold:
 \begin{enumerate}[(i)]
  \item $\ef$ is a nonempty family of nonempty $G_\delta$ subsets of $X$.
  \item For each $S\in\ef$ and each open set $G\subset X$ with $S\cap G\neq\emptyset$ there exists $S^*\in\ef$ such that $S^*\subset S\cap G$ and $S$ is $\langle g \rangle$-porous at no point of $S^*$.
 \end{enumerate}
\end{defin}

\begin{prop}[Foran Lemma] \label{FL}
Let $\langle X,\rho \rangle$ be a topologically complete metric space and let $\ef$ be a Foran system for $\langle g \rangle$-porosity in $X$. Then no member of $\ef$ is $\sigma$-$\langle g \rangle$-porous.
\end{prop}

This is a special case of the general Foran Lemma which works for any porosity-like relation (see [\ref{2}, Proposition 1]). Our definition of Foran system is, therefore, accordingly simplified as well. We also need the following.

\begin{notation}
By 3-porosity we mean $\langle g \rangle$-porosity where $g(x)=3x$ for $x \in \er$.
\end{notation}

\begin{lemma}[{[\ref{2}, Lemma E]}]\label{3iffNIC}
Let $\langle X,\rho \rangle$ be a metric space and let $A\subset X$. Then $A$ is $\sigma$-upper porous if and only if it is $\sigma$-3-porous.
\end{lemma}

Another result from \cite{2} which we shall use is the following partial converse of the Foran Lemma. For ordinary $\sigma$-upper porosity we can extend its validity from $G_\delta$ sets to Suslin sets using the inscribing Theorem \ref{inscrupper} of J.~Pelant and M.~Zelen\'y from the work \cite{pelzel}. 

It could be interesting to note that in case our metric space $X$ is locally compact, we can use a different inscribing theorem due to L.~Zaj\'i\v{c}ek and M.~Zelen\'y [\ref{zajzelinscr}, Theorem 5.2] and obtain an extension of \ref{Foran Converse} to analytic sets for general $\sigma$-$\langle g \rangle$-porosity.

\begin{lemma}[{[\ref{2}, Corollary 1]}]\label{Foran Converse}
Let $\langle X,\rho \rangle$ be a topologically complete metric space, let $\emptyset\neq A\subset X$ be $G_\delta$ and let $g$ be a porosity function. Then $A$ is not $\sigma$-$\langle g \rangle$-porous if and only if it contains a member of a Foran system for $\langle g \rangle$-porosity.
\end{lemma}

\begin{thm}[{[\ref{pelzel}, Theorem 3.1]}]\label{inscrupper}
Let $\langle X,\rho \rangle$ be a topologically complete metric space and let $S\subseteq X$ be a non-$\sigma$-upper porous Suslin set. Then there exists a closed non-$\sigma$-upper porous set $F\subseteq S$.
\end{thm}

\begin{defin}\label{dLower}
 Let $\langle X, \rho\rangle$ be a metric space, $A\subset X$ and $x \in X$. We say that $A$ is \emph{lower porous at} $x$ if $$\liminf_{R\to 0^+}\frac{\gamma(x,R,A)}{R} > 0.$$

The set $A$ is \emph{lower porous} if it is lower porous at each of its points and  \emph{$\sigma$-lower porous} if it is a countable union of lower porous sets.
\end{defin}

Even though the Foran Lemma can be used for any notion of porosity, we have to use a different approach in the case of lower porosity. The reason is that unlike in the case of upper porosity, we were unable to separably reduce the property of not being \emph{lower} porous at a point. Therefore, we use the following proposition.

\begin{prop}[{[\ref{martin}, Proposition 2.11]}]\label{cLowerIff}
Let $\langle X,\rho \rangle$ be a topologically complete metric space and let $A\subseteq X$ be a Suslin set. Then the following propositions are equivalent:
\begin{enumerate}[(i)]
\item A is not $\sigma$-lower porous.
\item There exists a closed set $F\subseteq A$ and a set $D\subseteq F$ dense in $F$ such that $F$ is lower porous at no point of $D$.
\end{enumerate}
\end{prop}

\section{Auxiliary results}

In this section we prove some preliminary statements which will be of use later. In general, for a space $X$ and a set $A\subset X$, we are trying to find a separable subspace $X_M\subset X$ with certain special properties. The first desired property is: Whenever $A$ is a member of a Foran system in $X$ then $A\cap X_M$ is a member of a Foran system in $X_M$. Together with Lemma \ref{Foran Converse} this will be essential to the proof of Theorem \ref{tUpPorous} about $\sigma$-upper porosity. 

Also, in order to prove a result similar to \ref{tUpPorous} for $\sigma$-lower porosity, two auxiliary propositions (based on the ideas from \cite{cuth}) are collected.

\begin{prop}\label{pGPorous}\pp\\ Let $\langle X,\rho\rangle$ be a metric space and $g$ a porosity function. Then whenever $M$ contains $X$ and a set $A\subset X$, it is true that for every $x\in X_M$
$$A\text{ is not }\langle g \rangle\text{-porous at }x \rightarrow A\cap X_M\text{ is not }\langle g \rangle\text{-porous at }x\text{ in the space $X_M$}.$$
If $M$ contains also $g$, then
$$A\text{ is not $\langle g \rangle$-porous} \rightarrow A\cap X_M\text{ is not $\langle g \rangle$-porous in the space $X_M$}.$$
\end{prop}
\begin{proof}
 Let us fix a $(*)$-elementary submodel $M$ containing $X$ and $A$ and fix some $x\in X_M$ such that $A$ is not $\langle g \rangle$-porous at $x$. Take sequences $\{c_n\}_{n\in\en}\subset X_M$ and $\{r_n\}_{n\in\en}\subset (0,\infty)$ such that $c_n\to x$ and $x\in U(c_n,g(r_n))$ for all $n\in\en$. It is sufficient to show that there exists an $n\in\en$ satisfying $U(c_n,r_n)\cap A\cap X_M\neq \emptyset$. Since $A$ is not $\langle g\rangle$-porous, we can fix some $n\in\en$ such that $U(c_n,r_n)\cap A\neq\emptyset$. Take some $a\in A\cap U(c_n,r_n)$ and choose an $\eps > 0$ such that $\rho(a,c_n) + 2\eps < r_n$. Then take a point $c\in X\cap M\cap U(c_n,\eps)$ and $q_n\in\qe\cap (\rho(a,c_n) + \eps, r_n - \eps)$. Hence,
 \begin{flalign*}(*) & & \exists a\in A\quad (\rho(a,c) < q_n).& &\end{flalign*}
 Thus, by the absoluteness, there exists an $a\in A\cap M$ such that
 $$\rho(a,c_n)\leq \rho(a,c) + \eps < q_n + \eps < r_n.$$
 Consequently, $a\in A\cap U(c_n,r_n)\cap M$ and thus the set $A\cap X_M$ is not $\langle g\rangle$-porous at $x$ in the space $X_M$.

 If $A$ is not $\langle g \rangle$-porous then
 \begin{flalign*}(*) & & \exists x\in A\quad(A\text{ is not }\langle g\rangle\text{-porous at }x).& &\end{flalign*}
Using the absoluteness and the already proved part we obtain an $x\in A\cap M$ such that $A\cap X_M$ is not $\langle g \rangle$-porous at $x$ in the space $X_M$.
\end{proof}

\begin{prop}\label{FS redukce}\pp\\ Let $\langle X,\rho\rangle$ be a topologically complete metric space and $g$ a porosity function. Then whenever $M$ contains $X$, $g$ and a set $A\subset X$, it is true that if $A$ is a member of a Foran system for $\langle g\rangle$-porosity in $X$, then $A\cap X_M$ is a member of a Foran system for $\langle g\rangle$-porosity in $X_M$.
\end{prop}
\begin{proof}
 Let us fix a $(*)$-elementary submodel $M$ containing $X$ such that $A\in M$ and let the following formula be true
  \begin{flalign*}(*) & & \exists\ef\;(\ef\text{ is a Foran system for $\langle g\rangle$-porosity in $X$ such that $A\in\ef$}).& &\end{flalign*}
Notice that the preceding is a formula with free variables from $M$. Thus, by the absoluteness, there exists an $\ef\in M$ which is a Foran system for $\langle g\rangle$-porosity in $X$ with $A\in\ef$. Set
$$\ef':=\{S\cap X_M:\;S\in \ef\cap M,\; S\cap X_M\neq\emptyset\}.$$
First we notice that, by Lemma \ref{lNonempty}, the set $A\cap M$ is nonempty; it follows that $A\cap X_M\in\ef'$. Thus it suffices to establish that $\ef'$ is a Foran system for $\langle g\rangle$-porosity in $X_M$. Clearly, $\ef'$ is a nonempty family of nonempty $G_\delta$ subsets of $X_M$ so there only remains to be verified the second condition from the definition of Foran system. 

To that end, take some $S\in\ef\cap M$ such that $S\cap X_M\neq\emptyset$ (denote by $S_M$ the set $S\cap X_M$). Then take an arbitrary open set $G\subset X$ with $S_M\cap G\neq\emptyset$ and fix some $x\in S_M\cap G$ and $r\in\qe_+$ such that $U(x,r)\subset G$. Choose $x_0\in U(x,\frac r2)\cap M$. Then $x\in U(x_0,\frac r2)\subset U(x,r)$; thus, $S\cap  U(x_0,\frac r2)\neq\emptyset$. Using Propositions \ref{pModelBasis} and \ref{pModelBasisSpaces} we obtain that $S\cap U(x_0,\frac r2)\in M$. 

Now, as $\ef$ is a Foran system (in $X$), the following formula is true:
\begin{flalign*}(*) & & \exists S^*\in\ef: (S^*\subset S\cap U(x_0,\tfrac r2), S\text{ is }\langle g\rangle\text{-porous at no point of $S^*$})& &\end{flalign*}
By the absoluteness, there exists an $S^*\in M$ satisfying the formula above. Using Lemma \ref{lNonempty} we can see that $S^*\cap M\neq\emptyset$. Thus, $S^*$ is a member of $\ef'$, $S^*\cap X_M\subset S_M\cap U(x_0,\frac r2)\subset S_M\cap G$ and by Proposition \ref{pGPorous} above, $S_M$ is $\langle g\rangle$-porous at no point of $S^*\cap X_M$. Consequently, $A\cap X_M$ is, indeed, a member of a Foran system for $\langle g\rangle$-porosity in $X_M$ -- the system $\ef'$.
\end{proof}

\begin{remark}
Note that the last proof depends solely on our ability to separably reduce $\langle g \rangle$-porosity of a set at a point. It would work for any other type of porosity which fulfils this condition, e.g., the $(g)$-porosity (for the definition see \cite{1} or \cite{7}).
\end{remark}

Before proceeding to the last section where we use the propositions above, let us briefly turn our attention to the matter of lower porosity and formulate two related facts:

\begin{lemma}\label{lPitomustka}\pp\\ Let $\langle X,\rho \rangle$ be a metric space, $A\subset X$ and $d(\cdot,A):X\to\er$ the function defined by the formula $d(\cdot,A)(x):=d(x,A)$. Then whenever $M$ contains $X$ and $A$ then $d(\cdot,A)$ is an element of $M$. 
\end{lemma}
\begin{proof}
 \letsMX such that $A\in M$. Then the lemma follows immediately from Lemma \ref{lUniqueInM} and from the absoluteness of the following formula and its subformulas \begin{flalign*}
  (*) & & \exists d(\cdot,A)\quad (d(\cdot,A)\text{ is a function which maps every } x\in X \text{ to the real number }& &\\
				       & & \inf\{\rho(x,a);\;a\in A\}). 
 \end{flalign*}
\end{proof}

Finally, we present the following proposition (its proof is contained in the proof of Proposition 4.10 in \cite{cuth}).

\begin{prop}\label{pLowerPoints}\pp\\ Let $\langle X,\rho \rangle$ be a metric space and $A\subset X$. Then whenever $M$ contains $X$ and $A$, it is true that for every $x\in A\cap M$
  $$A\text{ is not lower porous at $x$} \rightarrow A\cap X_M \text{ is not lower porous at } x \text{ in the space }X_M.$$ 
\end{prop}

Note, that this is exactly the moment, where we were unable to reduce the property of not being lower porous at a point. However, thanks to Proposition \ref{cLowerIff}, this proposition will be sufficient.

\section{Main results}

In the main part of this article we show that the set properties ``to be $\sigma$-upper porous'' and ``to be $\sigma$-lower porous'' are separably determined. We formulate the related theorems in the language of elementary submodels (which is useful when we want to combine several results concerning elementary submodels together). However, we also formulate a corollary of these results in such a setting that no knowledge of elementary submodels is required (see Theorem \ref{tUpperPorous}). 

Next, we show that these results may be useful for proving that some results concerning separable spaces hold in a nonseparable setting as well. This is demonstrated on Theorem \ref{tApplication}.

First, let us show that $\sigma$-upper porosity is a separably determined notion.

\begin{thm}\label{tUpPorous}\pp\\ Let $\langle X,\rho\rangle$ be a topologically complete metric space, $g$ a porosity function and $A\subset X$ a Suslin set. Then whenever $M$ contains $X$ and $A$, it is true that
$$A\text{ is $\sigma$-upper porous in } X \iff A\cap X_M\text{ is $\sigma$-upper porous in }X_M.$$
Moreover, if $A$ is $G_\delta$ and M contains also $g$, then
$$A\text{ is not $\sigma$-$\langle g \rangle$-porous in } X \rightarrow A\cap X_M\text{ is not $\sigma$-$\langle g \rangle$-porous in }X_M.$$
\end{thm}

\begin{proof}\letsMX such that $g,A\in M$. Assume the set $A$ is of the type $G_\delta$; we shall prove the second part of the proposition first. 
Due to Lemma \ref{Foran Converse} and the absoluteness of the formula (and its subformulas)
\begin{flalign*}
  (*) & & \exists B\quad (B\subset A\text{ and $B$ is a member of a Foran system for $\langle g \rangle$-porosity}), & & 
 \end{flalign*}
we can assume that the set $A$ is a member of a Foran system $\ef$ for $\langle g \rangle$-porosity. Hence the set $A\cap X_M$ is a member of a Foran system $\ef'$ for $\langle g \rangle$-porosity in $X_M$ (Proposition \ref{FS redukce}) and thus is not $\sigma$-$\langle g \rangle$-porous in $X_M$ (Proposition \ref{FL}).

The implication from the left to the right for $\sigma$-upper porosity follows immediately from Lemma \ref{lPitomustka} and [\ref{cuth}, Corollary 4.13]. 

We shall prove the other implication indirectly; owing to Theorem \ref{inscrupper} we can assume that $A$ is $G_\delta$ again (even closed). The result now follows from the already proved part and Lemma \ref{3iffNIC}, using the absoluteness of the formula (and its subformulas)
\begin{flalign*}(*) & & \exists g \quad(g:\er\to\er\text{ is a function such that for all }x\in\er\text{ is }g(x)=3x). & &\qedhere \end{flalign*}
\end{proof}

\begin{remark}
It is not known to the authors whether the other implication for $\sigma$-$\langle g\rangle$-porosity holds. However, under the assumptions of the preceding theorem, it is true that whenever $A$ is $\sigma$-$\langle g \rangle$-porous then $A\cap X_M$ is $\sigma$-$\langle dg \rangle$-porous in $X_M$ for any $d>2$. This may be established in the following way:

First, using the ideas presented in \cite{cuth} (mainly Proposition 4.12 and Corrolary 4.13), we are able to see that if $A$ is $\sigma$-$(g,c)$-porous in $X$ (where $c>0$; the definition is natural -- see \cite{1}), then $A\cap X_M$ is $\sigma$-$(g,c/2)$-porous in the space $X_M$. 

Now let us assume the set $A$ is $\sigma$-$\langle g \rangle$-porous in $X$. Then [\ref{1}, Lemma 3.1(ii)] implies it is $\sigma$-$(g,1/2)$-porous in $X$ and thus $A\cap X_M$ is $\sigma$-$(g,1/4)$-porous in the space $X_M$. In the nontrivial case when there exists a $\delta>0$ such that $g(x)>x$ for all $x\in (0,\delta)$ (if that is not the case, then the notion of $\langle g \rangle$-porosity is usually not very interesting) it is not difficult to prove that $g$ satisfies the assumption from [\ref{1}, Proposition 4.4]. Thus $A\cap X_M$ is $\sigma$-$(g,c)$-porous for any $c\in (0,1/2)$. To pass back to $\langle \cdot \rangle$-porosity, we use a slightly refined version of [\ref{1}, Lemma 3.1(i)] which for any $d>1$ states that $(f,d)$-porosity of a given set $N$ at a given point $x$ implies  $\langle f \rangle$-porosity of $N$ at $x$. We easily obtain that the set $A\cap X_M$ is $\sigma$-$\langle d g \rangle$-porous for any $d>2$.

Moreover, under the additional assumption that there exists a $d>2$ and a $\delta>0$ such that $g(x) > dx$ for any $x\in (0,\delta)$, we are able to prove (similarly as above) that whenever $A$ is $\sigma$-$\langle g \rangle$-porous then $A\cap X_M$ is $\sigma$-$\langle g \rangle$-porous in $X_M$.
\end{remark}
\begin{remark}
Under the assumptions of Theorem \ref{tUpPorous} the following holds: If $g$ is a porosity function such that for some $c>0$ there is a $\delta>0$ such that $cg(x)>x$ for all $x\in (0,\delta)$, then 
$$A\text{ is $\sigma$-$(g)$-porous in } X \iff A\cap X_M\text{ is $\sigma$-$(g)$-porous in }X_M.$$
This can be established as follows: Let $d=12c$ and let $A$ be non-$\sigma$-$(g)$-porous in $X$. Then it is non-$\sigma$-$(dg,1)$-porous and thus $A$ is non-$\sigma$-$\left\langle \frac{d}{2} g \right\rangle$-porous in $X$. Theorem \ref{tUpPorous} asserts that the same holds also for $A\cap X_M$ in $X_M$. Hence, $A\cap X_M$ is non-$\sigma$-$\left(\frac{d}{2}g,2\right)$-porous ([\ref{1}, Lemma 3.1(i)]), i.e., it is non $\sigma$-$\left(\frac{d}{12}g,\frac{1}{3}\right)$-porous (in $X_M$). Now, since the function $\frac{d}{12}g=cg$ satisfies the assumption of [\ref{1}, Proposition 4.4], we obtain that $A\cap X_M$ is not $\sigma$-$(g)$-porous in $X_M$.

For the proof of the other implication assume that the set $A$ is $\sigma$-$(g)$-porous. It is easy to see that there exist $(g,c_n)$-porous sets $A_n$ (with $c_n>0$ for each $n\in \en$) such that $A=\bigcup_{n=1}^\infty A_n$. In the same way as in the previous remark we obtain that $A_n\cap X_M$ is $\left( 	g,\frac{c_n}{2} \right)$-porous in $X_M$ for each $n$. Hence, $A\cap X_M$ is $\sigma$-$g$-porous in $X_M$.
\end{remark}


We shall now turn our attention to $\sigma$-lower porosity and show it is separably determined.

\begin{thm}\label{tLoPorous}\pp\\ Let $\langle X,\rho \rangle$ be a topologically complete metric space and let $A\subset X$ be a Suslin set. Then whenever $M$ contains $X$ and $A$, it is true that 
  $$A\text{ is $\sigma$-lower porous in }X \iff A\cap X_M\text{ is $\sigma$-lower porous in }X_M.$$ 
\end{thm}
\begin{proof}
 \letsMX such that $A\in M$. Then the implication from the left to the right follows from [\ref{cuth}, Corollary 4.13] and from Lemma \ref{lPitomustka}.
 
To prove the opposite implication we use Proposition \ref{cLowerIff}. Let us assume that the set $A$ is not $\sigma$-lower porous in $X$. Then \begin{flalign*}
   (*) & & \exists F\;\;\exists D \quad(\text{$F\subset A$ is a nonempty closed set such that $D\subset F$ is dense in $F$}& &\\
			& & \text{and $F$ is not lower porous at any point of $D$}).
 \end{flalign*}
 By the absoluteness of this formula (and its subformulas) above, we are able to find sets $F, D\in M$ satisfying the conditions above. Using Lemma \ref{lNonempty} and Lemma \ref{lDense} we can see that $F\cap M\neq\emptyset$ and $D\cap M$ is dense in $F\cap X_M$. Moreover, by Proposition \ref{pLowerPoints}, $F\cap X_M$ is not lower porous at any point of $D\cap M$. Thus, from Proposition \ref{cLowerIff} it follows that the set $A\cap X_M$ is not $\sigma$-lower porous in the space $X_M$.
\end{proof}

Theorem \ref{tUpperPorous} from the introduction is just an easy consequence of Theorem \ref{tUpPorous}, Theorem \ref{tLoPorous} and Proposition \ref{pModelBasisSpaces} since Convention \ref{conventionM} allows us to combine these three results; by doing that we obtain a theorem in the setting of Banach spaces which concerns both types of porosity.
In a similar way, Theorem \ref{tDifPorous} follows from the Theorem \ref{tUpPorous}, Theorem 5.7 in \cite{cuth} and Proposition \ref{pModelBasisSpaces} (because the set of the points where a function is Fr\'echet differentiable is a $F_{\sigma\delta}$ set - see for example [\ref{cuth}, Theorem 5.8])
).

Finally, we give the following application of our results. In \cite{5} is proved the following theorem (we use the more common terminology from \cite{kruger}).

\begin{defin}
 Let $\langle X,\left\| \cdot \right\| \rangle$  be a Banach space and let $f$ be a real function defined on $X$. We say that $f$ is Fr\'echet superdifferentiable at $x\in X$ \ifff there exists a $x^*\in X^*$ such that 
 $$\limsup_{h\to 0}\frac{(f(x+h) - f(x) - x^*(h))}{\|h\|}\leq 0.$$
\end{defin}

\begin{thm}[{[\ref{ZAJThm}, Theorem 2]}]\label{tSepDual}
Let $\langle X,\left\| \cdot \right\| \rangle$ be a Banach space with separable dual space and let $G\subset X$ be an open set. Let $f$ be a Lipschitz function on $G$ and let $A$ be the set of all the points $x\in G$ such that $f$ is Fr\'echet superdifferentiable at $x$ and $f$ is not Fr\'echet differentiable at $x$. Then $A$ is $\sigma$-upper porous.
\end{thm}

Using the method of elementary submodels, it is now easy to extend the validity of this result to general Asplund spaces.

\begin{thm}\label{tApplication}
 Let $\langle X,\left\| \cdot \right\| \rangle$ be an Asplund space and let $G\subset X$ be an open set. Let $f$ be a Lipschitz function on $G$ and let $A$ be the set of all the points $x\in G$ such that $f$ is Fr\'echet superdifferentiable at $x$ and $f$ is not Fr\'echet differentiable at $x$. Then $A$ is $\sigma$-upper porous.
\end{thm}
\begin{proof}
 Let us denote by $D(f)$ the set of points where $f$ is Fr\'echet differentiable and by $S(f)$ the set of points where $f$ is Fr\'echet superdifferentiable. Then, using Theorem \ref{tUpPorous}, Proposition \ref{pModelBasisSpaces} and  [\ref{cuth}, Theorem 5.7], take an elementary submodel $M$ satisfying:
 \begin{itemize}
  \item $X_M$ is a separable subspace of $X$,
  \item $D(f)\cap X_M = D(f\upharpoonright_{X_M})$,
  \item $A$ is $\sigma$-upper porous \ifff $A \cap X_M$ is $\sigma$-upper porous in the space $X_M$.
 \end{itemize}
 Note that $A\cap X_M \subset \{x\in X_M;x\in S(f\upharpoonright_{X_M})\setminus D(f\upharpoonright_{X_M})\}$ and that the set on the right side is $\sigma$-upper porous (because $X_M$ is a separable space with separable dual); thus the set $A$ is $\sigma$-upper porous.
\end{proof}

\end{document}